\documentclass[11pt]{amsart}

\usepackage{amssymb,amscd,amsthm, verbatim,amsmath,color,fancyhdr, mathrsfs, bm}
\usepackage{graphicx}
\usepackage{turnstile}

\usepackage{tikz,tikz-cd}

\usepackage{array}
\usepackage{multirow}

\usepackage{amssymb,epsfig,psfrag}
\usepackage{amsfonts, amssymb}
\usepackage{enumerate}
\usepackage[all]{xy}
\usepackage[linesnumbered,ruled]{algorithm2e}
\usepackage{setspace}
\usepackage{float}
\usepackage{amsfonts, amssymb}
\usepackage{amssymb,epsfig,psfrag}
\usepackage{here}
\usepackage{pict2e}
\usepackage{hyperref}
\usepackage{setspace}
\usepackage{url}
\usepackage{here}
\usepackage{pict2e}
\usepackage{hyperref}
\usepackage{blindtext}
\def\multiset#1#2{\ensuremath{\left(\kern-.3em\left(\genfrac{}{}{0pt}{}{#1}{#2}\right)\kern-.3em\right)}}

\hypersetup{
    colorlinks=true,
    linkcolor=blue,
    filecolor=magenta,      
    urlcolor=cyan,
}

\usepackage[letterpaper, left=2.5cm, right=2.5cm, top=2.5cm,
bottom=2.5cm,dvips]{geometry}

\newtheorem{theorem}{Theorem}
\newtheorem{lemma}[theorem]{Lemma}
\newtheorem{corollary}[theorem]{Corollary}

\theoremstyle{definition}

\theoremstyle{remark}

\title[Incidences between quadratic subspaces over finite fields]
{Incidences between quadratic subspaces \\over finite fields}

\author{Semin Yoo}

\date{\today}
\address{Department of Mathematics, University of Rochester, Rochester, NY, USA}

\email{syoo19@ur.rochester.edu}


\begin{document}

\maketitle

\begin{abstract}
Let $\mathbb{F}_{q}$ be a finite field of order $q$, where $q$ is an odd prime power. A quadratic subspace $(W,Q)$ of $(\mathbb{F}_{q}^{n},x_{1}^{2}+x_{2}^{2}+\cdots+x_{n}^{2})$ is called \textbf{dot$_{\mathbf{k}}$-subspace} if $Q$ is isometrically isomorphic to $x_{1}^{2}+x_{2}^{2}+\cdots+x_{k}^{2}$. In this paper, we obtain bounds for the number of incidences $I(\mathcal{K},\mathcal{H})$ between a collection $\mathcal{K}$ of dot$_{k}$-subspaces and a collection $\mathcal{H}$ of dot$_{h}$-subspaces when $h \geq 4k-4$, which is given by
\[\left | I(\mathcal{K},\mathcal{H})-\frac{|\mathcal{K}||\mathcal{H}|}{q^{k(n-h)}}\right |  \lesssim q^{\frac{k(2h-n-2k+4)+h(n-h-1)-2}{2}}\sqrt{|\mathcal{K}||\mathcal{H}|}. \]
In particular, we improve the error term in \cite{PTV} obtained by Phuong, Thang and Vinh (2019) for general collections of affine subspaces in the presence of our additional conditions.
\end{abstract}

\section{Introduction and statements of results}

\smallskip
Let $\mathcal{P}$ be a finite subset of $\mathbb{R}^{2}$, $\mathcal{L}$ a finite subset of lines in $\mathbb{R}^{2}$, and $I(\mathcal{P},\mathcal{L})$ the number of incidences between $\mathcal{P}$ and $\mathcal{L}$. In the field of combinatorial geometry, there is a famous question to ask how large $I(\mathcal{P},\mathcal{L})$ can be. In \cite{ST}, Szemerédi and Trotter (1983) prove that $I(\mathcal{P},\mathcal{L}) \lesssim |\mathcal{P}|^{2/3}|\mathcal{L}|^{2/3}+|\mathcal{P}|+|\mathcal{L}|$, where $x  \lesssim y$ denotes $x \leq cy$ for some constant $c$. 

\smallskip
The study of incidence problems over finite fields has been also conducted by a considerable number of researchers. In \cite{BKT}, Bourgain, Katz, and Tao (2004) show that for a given prime $p$ and $0< \alpha <2$, if $\mathcal{P}$ and $\mathcal{L}$ are sets of points and lines in $\mathbb{F}_{p}^{2}$, respectively, with $|\mathcal{P}|,|\mathcal{L}|<N=p^{\alpha}$, then there exists $\epsilon(\alpha) >0$ such that $I(\mathcal{P},\mathcal{L})\lesssim N^{\frac{3}{2}-\epsilon(\alpha)}$.

\smallskip
In particular, in \cite{PTV}, Phuong, Thang and Vinh (2019) recently obtain bounds for the number of incidences between $k$-planes and $h$-planes in $\mathbb{F}_{q}^{n}$. They obtain the main term and bounds on the error term of incidences by using spectral gap theorem and graph theory. Let $V$ be a subset in $\mathbb{F}_{q}^{n}$. We call $V$ an \textbf{affine} $\mathbf{k}$\textbf{-plane} in $\mathbb{F}_{q}^{n}$, where $k<d$, if there exist $k+1$ vectors $v_{1},v_{2},\cdots,v_{k+1}$ in $\mathbb{F}_{q}^{n}$ such that
\[V=\text{span}\left \{ v_{1},\cdots,v_{k} \right \}+v_{k+1},\quad \text{rank}\left \{ v_{1},\cdots,v_{k} \right \}=k.\]
The statement of the main theorem in \cite{PTV} is as follows.
\begin{theorem}[Phuong, Thang and Vinh '19]\label{ptv}
Let $\mathcal{P}$ be a set of affine $k$-planes and $\mathcal{H}$ be a set of affine $h$-planes in $\mathbb{F}_{q}^{n}$ ($h \geq 2k+1$). Then
\[\left | I(\mathcal{P},\mathcal{H})-\frac{|\mathcal{P}||\mathcal{H}|}{q^{(k+1)(n-h)}}\right | \lesssim q^{\frac{k(2h-n-k+1)+h(n-h)}{2}}\sqrt{|\mathcal{P}||\mathcal{H}|} .\]
\end{theorem}
\noindent More history of the improvements for the incidence problem over finite fields can be found in \cite{PTV}. 

\smallskip
In this paper, motivated by their work, we obtain bounds on the number of incidences between dot$_{k}$-subspaces and dot$_{h}$-subspaces. Our error term is less than the one in Theorem \ref{ptv}. Recall that any nondegenerate quadratic form on a vector space of dimension $n$ over a finite field is equivalent to one of the following two (see \cite{Co}):
\begin{align*}
\text{dot}_n(\mathbf{x}) & :=x_{1}^{2}+\cdots+x_{n-1}^{2}+x_{n}^{2} \quad \text{or}\\
\text{$\lambda$dot$_n$}(\mathbf{x}) & :=x_{1}^{2}+\cdots+x_{n-1}^{2}+\lambda x_{n}^{2} \quad \text{for some nonsqaure $\lambda$ in $\mathbb{F}_{q}$}.
\end{align*} 
When there is no danger of confusion, we will let dot$_{n}$ and $\lambda$dot$_{n}$ denote the quadratic spaces
$(\mathbb{F}_q^n, \text{dot}_n(\mathbf{x}))$ and $(\mathbb{F}_q^n, \lambda$dot$_n(\mathbf{x}))$ respectively. By the classification, possible $k$-dimensional quadratic subspaces of $(\mathbb{F}_{q}^{n},\text{dot}_{n}(\mathbf{x}))$ are following:
\begin{align*}
&\text{dot}_{k},\text{dot}_{k-1}\oplus 0,\cdots,\text{dot}_{1}\oplus 0^{k-1}\\
&\lambda\text{dot}_{k},\lambda\text{dot}_{k-1}\oplus 0,\cdots,\lambda\text{dot}_{1}\oplus 0^{k-1}\\
&0^{k}
\end{align*}
It is not true in general that all those types always exist in the subspaces of $(\mathbb{F}_{q}^{n},\text{dot}_{n}(\mathbf{x}))$. The existence of possible quadratic subspaces of $(\mathbb{F}_{q}^{n},\text{dot}_{n}(\mathbf{x}))$ can in found in \cite{Yo1}. Among such many types of quadratic subspaces of $(\mathbb{F}_{q}^{n},\text{dot}_{n}(\mathbf{x}))$, we are mainly interested in one type of quadratic subspaces of $(\mathbb{F}_{q}^{n},\text{dot}_{n}(\mathbf{x}))$, which is dot$_{k}$. A quadratic subspace $(W,Q)$ of $(\mathbb{F}_{q}^{n},\text{dot}_{n}(\mathbf{x}))$ is called \textbf{dot$_{\mathbf{k}}$-subspace} if $Q$ is isometrically isomorphic to dot$_{k}(\mathbf{x})$. For later work, we also define a $\bm{\lambda}$\textbf{dot}$_{\mathbf{k}}$\textbf{-subspace} $(W,Q)$ if $Q$ is isometrically isomorphic to $x_{1}^{2}+\cdots+x_{k-1}^{2}+\lambda x_{k}^{2}$ for some nonsquare $\lambda$.

\smallskip
It turns out that the number of dot$_{k}$-subspaces has an interesting combinatorial description which can be describe as the form of the analogue of binomial coefficients. For more information, see \cite{Yo1}. Let us denote it by $\binom{n}{k}_{d}$. Also, in \cite{Yo2}, the author studies the graph coming from the set of dot$_{k}$-subspaces and the adjacency of edges is made by orthogonality between dot$_{k}$-subspaces. See \cite{Yo2}.

\smallskip
The statement of our main theorem is the following:
\begin{theorem}\label{main}
Let $\mathcal{K}$ be a collection of dot$_{k}$-subspaces and $\mathcal{H}$ be a collection of dot$_{h}$-subspaces in $(\mathbb{F}_{q}^{d},\text{dot}_{n})$, $1<k<h$ and $h \geq 4k-4$. Then we have
\[\left | I(\mathcal{K},\mathcal{H})-\frac{|\mathcal{K}||\mathcal{H}|}{q^{k(n-h)}}\right |  \lesssim q^{\frac{k(2h-n-2k+4)+h(n-h-1)-2}{2}}\sqrt{|\mathcal{K}||\mathcal{H}|} .\]
\end{theorem}
It is easy to check that sets $\mathcal{K},\mathcal{H}$ satisfying the inequality are nonempty. Furthermore, it follows from Theorem \ref{main} that 
\begin{corollary} With the same notation in Theorem \ref{main}, 
\begin{enumerate}
    \item If $|\mathcal{K}||\mathcal{H}|  \gtrsim q^{k(2h-n-2k+4)+h(n-h-1)-2} q^{2k(n-h)}$, then $I(\mathcal{K},\mathcal{H})$ is nonempty.
    \item If $|\mathcal{K}||\mathcal{H}|  \gg q^{k(2h-n-2k+4)+h(n-h-1)-2}q^{2k(n-h)}$, then $I(\mathcal{K},\mathcal{H})$ is close to the main term $|\mathcal{K}||\mathcal{H}|/q^{k(n-h)}$,
\end{enumerate}
where $x \gg y$ means $y=o(x)$.
\end{corollary}

The main difficulty of this work comes from the types of quadratic subspaces. For example, the intersection between dot$_{k}$-subspaces may not be a dot$_{k}$-subspace again. When we bound the third eigenvalue of the adjacency matrix of the incidence graph made by dot$_{k}$-subspaces and dot$_{h}$-subspaces, we avoid this difficulty by defining a certain graph using sums of vector subspaces of $\mathbb{F}_{q}^{n}$ to disregard the types of degenerate quadratic subspaces of $(\mathbb{F}_{q}^{n},\text{dot}_{n}(\mathbf{x}))$, and use some results from \cite{Yo1}.

\smallskip
\subsection*{Acknowledgements.} 
The author would like to express gratitude to Jonathan Pakianathan for helpful discussions and encouragement for this work.

\medskip
\section{Preliminaries}
\smallskip

In this section, we remind our reader of some facts coming from the theory of quadratic forms, spectral graph theory that we will need later. First, we introduce the following two equivalent fundamental facts in the theory of quadratic forms. In this paper, we will collectively call any of them Witt's theorem. The proofs can be found in \cite{Cl}.

\begin{theorem}[Witt's extension theorem]\label{we}\cite{Cl}
Let $X_{1},X_{2}$ be quadratic spaces and $X_{1} \cong X_{2}$, $X_{1}=U_{1} \oplus V_{1},X_{2}=U_{2}\oplus V_{2}$, and $f:V_{1} \longrightarrow V_{2}$ an isometry. Then there is an isomtery $F:X_{1}\longrightarrow X_{2}$ such that $F|_{V_{1}}=f$ and $F(U_{1})=U_{2}$.
\end{theorem}

\begin{theorem}[Witt's cancellation theorem]\label{wc}\cite{Cl}
Let $U_{1},U_{2},V_{1},V_{2}$ be quadratic spaces where $V_{1}$ and $V_{2}$ are isometrically isomorphic. If $U_{1}\oplus V_{1} \cong U_{2} \oplus V_{2}$, then $U_{1} \cong U_{2}$.
\end{theorem}

We also rely on the following theorem from spectral graph theory, called spectral gap theorem or expanding mixing lemma. \begin{lemma}\cite{PTV} \label{spectral}
Suppose that $G$ is biregular with parts $A,B$ s.t deg$(A)=a$, deg$(B)=b$. Let $X \subset A$, $Y \subset B$ and $e(X,Y)$ be the number of incidences between $X$ and $Y$. Then we have
\[\left | I(X,Y)-\frac{a}{|B|}|X||Y| \right | \leq \lambda_{3} \sqrt{|X||Y|},\]
where $\lambda_{3}$ is the third eigenvalue of the adjacency matrix of $G$.
\end{lemma}
The rough sketch of the proof of Lemma \ref{spectral} is following. We follow \cite{PTV}. Let $A(G)$ be the adjacency matrix of $G$. Then it is written by
\[A(G)=\begin{pmatrix}
0 & N\\ 
N^{T} & 0
\end{pmatrix},\]
where $N_{ij}=1$ if there is an edge between $i$ and $j$. Let us label the eigenvalues satisfying $|\lambda_{1}| \geq |\lambda_{2}|\geq \cdots \geq |\lambda_{n}|$. Note that the number of walks of length $2$ in $G$ is $ab$. Thus $A^{2}(G)$ is regular and row sums of $A^{2}(G)$ are $ab$. This means that $ab$ is an eigenvalue of $A^{2}(G)$ corresponding to the vector with entries all $1$. By Perron-Frobenius Theorem (See \cite{GR}), we have $\lambda_{1}=\sqrt{ab}$ and $\lambda_{2}=-\sqrt{ab}$. We can also check $\sqrt{a}1_{A}\pm \sqrt{b}1_{B}$ are eigenvectors with the eigenvalues $\pm \sqrt{ab}$ respectively, where $1_{A}$ means the column vector which has $1$s in the positions corresponding to $A$ and $0$s otherwise. Let $W^{\perp}$ be the vector space spanned by $1_{A}$ and $1_{B}$. Since eigenvectors are orthogonal by the fact that $A(G)$ is symmetric, the other eigenvectors except for $1_{A}$ and $1_{B}$ span $W$. Thus $Ku=0$ for any $u$ in $W$, where 
\[K=\begin{pmatrix}
0 & J\\ 
J^{T} & 0
\end{pmatrix},\] where $J$ is the $|A|\times |B|$ matrix with entries all $1$. Also, the biggest eigenvalue corresponding to other eigenvectors is $\lambda_{3}$. Thus we have $||A(G)(1_{A})|| \leq \lambda_{3} ||1_{A}||$ and $||A(G)(1_{B})|| \leq \lambda_{3} ||1_{B}||$. By use this inequality, we can show that
\[\left | I(X,Y)-\frac{a}{|B|}|X||Y| \right | \leq \lambda_{3}||\overline{1_{X}}||||\overline{1_{Y}}|| ,\]
where $\overline{1_{X}}$ and $\overline{1_{Y}}$ are projection of $1_{X}$ and $1_{Y}$ onto $W$. For more details, see \cite{PTV}. It is not hard to show $||\overline{1_{X}}||||\overline{1_{Y}}|| \leq \sqrt{|X||Y|}$. This completes the proof.

\smallskip
Let us consider $A^{2}(G)$. Then we have
\[A^{2}(G)=\begin{pmatrix}
NN^{T} & 0\\ 
0 & N^{T}N
\end{pmatrix}.\]
Suppose that $v_{3}=(v_{1},\cdots,v_{m},u_{1},\cdots,u_{n})^{T}$ be an eigenvector of $A(G)$ corresponding to the eigenvalue $\lambda_{3}$. It is easy to check that $v_{3}$ is an eigenvector of $A^{2}(G)$ with the eigenvalue $\lambda_{3}^{2}$ and thus $(v_{1},\cdots,v_{m})^{T}$ is an eigenvector of $NN^{T}$ with the eigenvalue $\lambda_{3}^{2}$. By the proof of Theorem \ref{spectral}, $Kv_{3}=0$. This implies $J(v_{1},\cdots,v_{m})^{T}=0$. Therefore, we summarize our discussion in the following Lemma.
\begin{lemma}\cite{PTV}\label{eigen}
Let $G$ be a biregular graph with parts $A,B$ with $|A|=m,|B|=n$. Let $A(G)$ be the adjacency matrix of $G$ given by $A(G)=\bigl(\begin{smallmatrix}
0 & N\\ 
N^{T} & 0
\end{smallmatrix}\bigr)$. Let $v_{3}=( v_{1},\cdots,v_{m},u_{1},\cdots,u_{n} )^{T}$ be an eigenvector of $A(G)$ to $\lambda_{3}$. Then we have
\begin{enumerate}
    \item $(v_{1},\cdots,v_{m})^{T}$ is an eigenvector of $NN^{T}$,
    \item $J(v_{1},\cdots,v_{m})^{T}=0$.
\end{enumerate}
\end{lemma}

\smallskip
\section{Incidence graph coming from quadratic spaces.}
Let us define an incidence graph $G=(A \cup B, E)$, where $A$ is the set of dot$_{k}$-subspaces, $B$ is the set of dot$_{h}$-subspaces and $k < h$. For any dot$_{k}$-subspace $K$ in $A$ and dot$_{h}$-subspaces $H$ in $B$, there is an edge between them if $K \subset H$. By section 2 and 3 of \cite{Yo1}, the number of dot$_{k}$-subspaces of dot$_{n}$, denoted by $\binom{ n}{k}_{d}$, is given as follows:
\[\binom{n}{k}_{d}=(1+o(1))\frac{q^{k(n-k)}}{2}.\] 
Therefore, the sizes of $A$ and $B$ are shown to be
\[|A|=(1+o(1))\frac{q^{k(n-k)}}{2} ~\text{ and }~ |B|=(1+o(1))\frac{q^{h(n-h)}}{2}\]
respectively. Furthermore, we mention the number of $\lambda$dot$_{k}$-subspaces for later work. It is obtained by 
\[\binom{n}{\lambda k}_{d}=(1+o(1))\frac{q^{k(n-k)}}{2}.\]
Let $A(G)$ be the adjacency matrix of $G$. Since $G$ is bipartite, $A(A)$ is written by following:
\[A(G)=\begin{pmatrix}
0 & N \\ 
N^{T} & 0
\end{pmatrix},\]
where $N$ is a $q^{k(n-k)} \times q^{h(n-h)}$ matrix which is $1$ if $K \subset H$ and $0$ otherwise. By Witt's theorem, any vertex in $A$ has the same degree. Similarly, so does $B$. Thus $G$ is a biregular graph. We now count the degree of a dot$_{k}$-subspace in $G$.
\begin{lemma}\label{dega}
Let $n$ and $h$ be positive integers. Then the number of dot$_{h}$-subspaces containing a fixed dot$_{k}$ subspace is given by
\[\binom{n-k}{h-k}_{d}=\frac{1}{2}q^{(h-k)(n-h)}.\]
\end{lemma}
\begin{proof}
Let $K$ be a dot$_{k}$-subspace of dot$_{n}$. By Witt's theorem, we may choose any fixed dot$_{k}$-subspace. Let $\mathcal{V}$ be the set of dot$_{n-k}$-subspaces in dot$_{n}/K$ and $\mathcal{W}$ be the set of dot$_{k}$-subspace in dot$_{n}$ containing $K$. Define a function $f:\mathcal{V} \longrightarrow \mathcal{W}$ given by $W  \mapsto  W \oplus K$. It is not hard to show that $f$ is bijective. 
\end{proof}
Next, we consider the square of the adjacency matrix $A^{2}(G)$ given by
\[A^{2}(G)=\begin{pmatrix}
NN^{T} & 0\\ 
0 & N^{T}N
\end{pmatrix} .\]
Elements of $A^{2}(G)$ consist of two types of walks of length $2$: 
\begin{itemize}
    \item $a$:= the number of walks of length $2$ with the same starting and the end point,
    \item $b$:= the number of walks of length $2$ with the starting and the end point being different.
\end{itemize}
Let us define a new graph $E_{t}$ for each $t$ such that $k+1 \leq t \leq 2k \leq n$. The vertex set of $E_{t}$ is the set of dot$_{k}$-subspaces, and there is an edge between two distinct dot$_{k}$-subspaces $K,K'$ if $K+K'\simeq \mathbb{F}_{q}^{t}$. The adjacency matrix of $E_{t}$ is shown to be 
\[(E_{t})_{KK'}=\begin{cases}
1 & \text{ if } K+K'\simeq \mathbb{F}_{q}^{t} \\ 
0 & \text{otherwise} 
\end{cases},\]
where $+$ is a sum of subspaces of $\mathbb{F}_{q}^{t}$. Then we have
\begin{equation}\label{eq}
NN^{T}=aI+\sum_{t=k+1}^{2k}b_{t}E_{t}=\frac{1}{2}q^{(h-k)(n-h)}I+\sum_{t=k+1}^{2k}b_{t}E_{t},
\end{equation}
where $b_{k+1}+\cdots+b_{2k}=b$. 

\begin{lemma}\label{mc}
Let two distinct dot$_{k}$-subspaces $K$ and $K'$ be given. For each $t$ such that $k+1 \leq t \leq 2k$, the maximal degree of $(E_{t})_{KK'}$ happens when $K+K'\simeq (\mathbb{F}_{q}^{t},\text{dot}_{t})$ or $K+K'\simeq (\mathbb{F}_{q}^{t},\lambda \text{dot}_{t})$, and the degree of $(E_{t})_{KK'}$ is $\frac{1}{2}q^{(t-k)(n+2k-2t)}$ if $K+K'\simeq (\mathbb{F}_{q}^{t},\text{dot}_{t})$ or $K+K'\simeq (\mathbb{F}_{q}^{t},\lambda \text{dot}_{t})$.
\end{lemma}

\begin{proof}
Fix a dot$_{k}$-subspace $K$ and we count the number of $K'$ satisfying $K + K' \simeq (\mathbb{F}_{q}^{t},Q)$ for a fixed $Q$. First, suppose that $Q=\text{dot}_{t}$. Since $K$ is nondegenerate and $k$-dimensional, such a $K'$ is equivalent to choosing a $(t-k)$-dimensional nondegenerate quadratic subspaces $H$ such that $K \oplus  H=(\mathbb{F}_{q}^{t},\text{dot}_{t})$ as a quadratic subspace, and choosing a $(2k-t)$-dimensional subspace of $K$ as $K \cap K'$.

\smallskip
(1) The number of choices of dot$_{t-k}$-subspaces for $H$ is $\binom{n-k}{t-k}_{d}=\frac{1}{2}q^{(t-k)(n-t)}$, and the number of choices of dot$_{2k-t}$-subspace in $K$ is $\binom{k}{2k-t}_{d}=\frac{1}{2}q^{(2k-t)(t-k)}$. Thus, we have
\[\binom{n-k}{t-k}_{d}\binom{k}{2k-t}_{d}=\frac{1}{4}q^{(t-k)(n-t)}q^{(2k-t)(t-k)}=\frac{1}{4}q^{(t-k)(n+2k-2t)}.\]
(2) The number of choices of $\lambda$dot$_{t-k}$-subspaces for $H$ is $\binom{n-k}{\lambda (t-k)}_{d}=\frac{1}{2}q^{(t-k)(n-t)}$, and the number of choices of $\lambda$dot$_{2k-t}$-subspace in $K$ is $\binom{k}{\lambda(2k-t)}_{d}=\frac{1}{2}q^{(2k-t)(t-k)}$. Thus, we have
\[\binom{n-k}{\lambda(t-k)}_{d}\binom{k}{\lambda(2k-t)}_{d}=\frac{1}{4}q^{(t-k)(n-t)}q^{(2k-t)(t-k)}=\frac{1}{4}q^{(t-k)(n+2k-2t)}.\]
By (1) and (2), the number of $K'$ such that $K+K'\simeq (\mathbb{F}_{q}^{t},\text{dot}_{t})$ is
\begin{align*}
\binom{n-k}{t-k}_{d}\binom{k}{2k-t}_{d}+\binom{n-k}{\lambda(t-k)}_{d}\binom{k}{\lambda(2k-t)}_{d}&=\frac{1}{4}q^{(t-k)(n+2k-2t)}+\frac{1}{4}q^{(t-k)(n+2k-2t)}\\
&=\frac{1}{2}q^{(t-k)(n+2k-2t)}.
\end{align*}
Next, suppose that $Q=\lambda \text{dot}_{t}$. Similar way with the case of $Q=\text{dot}_{t}$, the number of $K'$ such that $K+K'\simeq (\mathbb{F}_{q}^{t},\lambda \text{dot}_{t})$ is
\begin{align*}
\binom{n-k}{\lambda(t-k)}_{d}\binom{k}{2k-t}_{d}+\binom{n-k}{t-k}_{d}\binom{k}{\lambda(2k-t)}_{d}&=\frac{1}{4}q^{(t-k)(n+2k-2t)}+\frac{1}{4}q^{(t-k)(n+2k-2t)}\\
&=\frac{1}{2}q^{(t-k)(n+2k-2t)}.
\end{align*}
Assume that $Q$ is degenerate. By \cite{Yo1}, most of the types of quadratic subspaces are nondegenerate since
\[\lim_{q\rightarrow \infty}\frac{\binom{n}{k}_{d}}{\binom{n}{k}_{q}}=\frac{1}{2},~~~\text{and also }~~\lim_{q\rightarrow \infty}\frac{\binom{n}{\lambda k}_{d}}{\binom{n}{k}_{q}}=\frac{1}{2}.\]
Here $\binom{n}{k}_{q}$ is the number of $k$-dimensional subspaces of $\mathbb{F}_{q}^{n}$. Therefore, this completes the proof.
\end{proof}

\begin{lemma}\label{bt}
Let $t$ be a nonnegative integer such that $k+1 \leq t \leq 2k \leq n$. Then the number of dot$_{h}$-subspaces containing two fixed dot$_{k}$-subspaces $K$ and $K'$ such that $K+K'\simeq (\mathbb{F}_{q}^{t},\text{dot}_{t})$ or $K+K'\simeq (\mathbb{F}_{q}^{t},\lambda \text{dot}_{t})$ is 
\[b_{t}=\frac{1}{2}q^{(h-t)(n-h)}.\]
\end{lemma}

\begin{proof}
The number of dot$_{h}$-subspaces containing two fixed dot$_{k}$-subspaces $K$ and $K'$ such that $K+K'\simeq (\mathbb{F}_{q}^{t},\text{dot}_{t})$ is
\[\binom{n-t}{h-t}_{d}=\frac{1}{2}q^{(h-t)(n-h)}.\]
Similarly, the number of dot$_{h}$-subspaces containing two fixed dot$_{k}$-subspaces $K$ and $K'$ such that $K+K'\simeq (\mathbb{F}_{q}^{t},\lambda \text{dot}_{t})$ is
\[\binom{n-t}{\lambda (h-t)}_{d}=\frac{1}{2}q^{(h-t)(n-h)}.\]
\end{proof}

\smallskip
Our next task is to bound the third eigenvalue of $A(G)$. 

\begin{theorem}\label{third}
The third eigenvalue of $A(G)$ is bounded by $\sqrt{\frac{k}{2}}q^{(-2k^{2}+2hk+4k-kn-h+hn-h^{2}-2)/2}$.
\end{theorem}

\begin{proof}
By Lemma \ref{eigen}, it suffices to bound the second eigenvalue of $NN^{T}$. By Lemma \ref{dega} and Lemma \ref{bt}, we have
\[NN^{T}=\frac{1}{2}q^{(h-k)(n-h)}I+\sum_{t=k+1}^{2k}q^{(h-t)(n-h)}E_{t}.\]
To obtain a better bound, we want to rewrite $NN^{T}$ with the matrix $J$.
\[
NN^{T}=\frac{1}{2}q^{(h-2k)(n-h)}J+\frac{1}{2}\left (q^{(h-k)(n-h)}-q^{(h-2k)(n-h)}\right )I+\frac{1}{2}\sum_{t=k+1}^{2k-1}\left (q^{(h-t)(n-h)}-q^{(h-2k)(n-h)}\right )E_{t}.
\]
Let $v_{3}$ be an third eigenvector of $A(G)$ with the eigenvalue $\lambda_{3}$. By Lemma \ref{eigen}, we have 
\[\lambda_{3}^{2}v_{3} = \frac{1}{2}\left (q^{(h-k)(n-h)}-q^{(h-2k)(n-h)}\right )v_{3}+\frac{1}{2}\left (\sum_{t=k+1}^{2k-1}(q^{(h-t)(n-h)}-q^{(h-2k)(n-h)})E_{t}  \right )v_{3}.\]
Thus $v_{3}$ is an eigenvector of 
\[\sum_{t=k+1}^{2k-1}(q^{(h-t)(n-h)}-q^{(h-2k)(n-h)})E_{t}.\]
Now, let us bound $\lambda_{3}^{2}$. Since the sums of eigenvalues are bounded by the sums of the largest eigenvalue, and by Lemma \ref{mc} we have
\begin{align*}
\lambda_{3}^{2} & \leq \frac{1}{2}q^{(h-k)(n-h)}+k\cdot \text{max}_{t,Q}\left | \frac{1}{2}q^{(h-t)(n-h)}E_{t} \right |\\
& \leq \frac{1}{2}q^{(h-k)(n-h)}+\frac{k}{4}\cdot \text{max}_{t} q^{(h-t)(n-h)}q^{(t-k)(n+2k-2t)}\\
&\leq \frac{1}{2}q^{(h-k)(n-h)}+\frac{k}{4}\cdot \text{max}_{t}q^{-2t^{2}+(4k+h)t+hn-h^{2}-kn-2k^{2}}\\
&=\frac{1}{2}q^{(h-k)(n-h)}+\frac{k}{4}q^{-2k^{2}+2hk+4k-kn-h+hn-h^{2}-2}.
\end{align*}
The maximum attains at $t=2k-1$ when $h \geq 4k-4$. If $k>1$, then we have
\[\frac{1}{2}q^{(h-k)(n-h)}\ll \frac{k}{4}q^{-2k^{2}+2hk+4k-kn-h+hn-h^{2}-2}.\]
In other words, we obtain
\[\lambda_{3} \leq \sqrt{\frac{k}{2}}q^{\frac{-2k^{2}+2hk+4k-kn-h+hn-h^{2}-2}{2}}.\]
\end{proof}
Finally, we finish to prove our main theorem.
\begin{proof}[Proof of Theorem \ref{main}]
By Lemma \ref{dega} and the size of $B$, we have
\[\frac{\text{deg}(A)}{|B|}=(1+o(1))q^{k(n-h)}.\]
Hence, Theorem \ref{spectral} and Theorem \ref{third} complete the proof.
\end{proof}

\end{document}